\documentclass[11pt]{article}

\usepackage[a4paper, margin=1in]{geometry}
\usepackage{amsfonts,amsmath,amssymb,amsthm,graphicx,fixmath,latexsym,color}
\usepackage{color}
\usepackage{xcolor}
\usepackage{hyperref,epsfig}
\usepackage{algorithmic}
\usepackage{algorithm}
\usepackage{xspace}

\def\eps{{\epsilon}}
\providecommand{\remove}[1]{}

\newcommand{\HD}{\ensuremath{\mathsf{HD}}\xspace}

\renewcommand{\Re}{\mathbb{R}}

\newcommand{\G}{\mathcal{G}}

\newcommand{\F}{\mathcal{F}}

\newtheorem{theorem}{Theorem}[section]
\newtheorem{lemma}[theorem]{Lemma}
\newtheorem{proposition}[theorem]{Proposition}
\newtheorem{claim}[theorem]{Claim}
\newtheorem{corollary}[theorem]{Corollary}

\newtheorem{remark}[theorem]{Remark}

\newtheorem*{theorem*}{Theorem}
\newtheorem*{lemma*}{Lemma}
\newtheorem*{proposition*}{Proposition}
\newtheorem{observation}[theorem]{Observation}

\begin{document}

\title{From a $(p,2)$-Theorem to a Tight $(p,q)$-Theorem}

\author{Chaya Keller\thanks{Department of Mathematics, Ben-Gurion University of the NEGEV, Be'er-Sheva Israel. \texttt{kellerc@math.bgu.ac.il}. Research partially supported by Grant 635/16 from the Israel Science Foundation, by the Shulamit Aloni Post-Doctoral Fellowship of the Israeli Ministry of Science and Technology, and by the Kreitman Foundation Post-Doctoral Fellowship.}
\and
Shakhar Smorodinsky\thanks{Department of Mathematics, Ben-Gurion University of the NEGEV, Be'er-Sheva Israel. \texttt{shakhar@math.bgu.ac.il}. Research partially supported by Grant 635/16 from the Israel Science Foundation.}
}

\date{}
\maketitle

\begin{abstract}
A family $\F$ of sets is said to satisfy the $(p,q)$-property if among any $p$ sets of $\F$ some $q$ have a non-empty intersection. The celebrated $(p,q)$-theorem of Alon and Kleitman asserts that any family of compact convex sets in $\Re^d$ that satisfies the $(p,q)$-property for some $q \geq d+1$, can be pierced by a fixed number (independent on the size of the family) $f_d(p,q)$ of points. The minimum such piercing number is denoted by $\HD_d(p,q)$. Already in 1957, Hadwiger and Debrunner showed that whenever $q>\frac{d-1}{d}p+1$ the piercing number is $\HD_d(p,q)=p-q+1$; no exact values of $\HD_d(p,q)$ were found ever since.

While for an arbitrary family of compact convex sets in $\Re^d$, $d \geq 2$, a $(p,2)$-property does not imply a bounded piercing number, such bounds were proved for numerous specific families. The best-studied among them is axis-parallel boxes in $\Re^d$, and specifically, axis-parallel rectangles in the plane. Wegner and (independently) Dol'nikov used a $(p,2)$-theorem for axis-parallel rectangles to show that $\HD_{\mathrm{rect}}(p,q)=p-q+1$ holds for all $q>\sqrt{2p}$. These are the only values of $q$ for which $\HD_{\mathrm{rect}}(p,q)$ is known exactly.

In this paper we present a general method which allows using a $(p,2)$-theorem as a bootstrapping to obtain a tight $(p,q)$-theorem, for families with Helly number 2, even without assuming that the sets in the family are convex or compact. To demonstrate the strength of this method, we obtain a significant improvement of an over 50 year old result by Wegner and Dol'nikov. Namely, we show that $\HD_{\mathrm{d-box}}(p,q)=p-q+1$ holds for all $q > c' \log^{d-1} p$, and in particular, $\HD_{\mathrm{rect}}(p,q)=p-q+1$ holds for all $q \geq 7 \log_2 p$ (compared to $q \geq \sqrt{2p}$ of Wegner and Dol'nikov).


In addition, for several classes of families, we present improved $(p,2)$-theorems, some of which can be used as a bootstrapping to obtain tight $(p,q)$-theorems. In particular, we show that any family $\F$ of compact convex sets in $\Re^d$ with Helly number 2 admits a $(p,2)$-theorem with piercing number $O(p^{2d-1})$, and thus, satisfies $\HD_{\F}(p,q)=p-q+1$ for all $q>cp^{1-\frac{1}{2d-1}}$, for a universal constant $c$.
\end{abstract}


\section{Introduction}

\subsection{Helly's theorem and (p,q)-theorems}

The classical Helly's theorem says that if in a family of compact convex sets in $\Re^d$
every $d+1$ members have a non-empty intersection then the whole family has a non-empty intersection.

For a pair of positive integers $p \geq q$, we say that a family $\F$ of sets satisfies the $(p,q)$-property if $|\F|\ge p$, none of the sets in $\F$ is empty, and among any $p$ sets of $\F$ there are some $q$ with a non-empty intersection. A set $P$ is called a {\em transversal} (or alternatively, a {\em piercing set}) for $\F$ if it has a non-empty intersection with every member of $\F$. In this language, Helly's theorem states that any family of compact convex sets in $\Re^d$ satisfying the $(d+1,d+1)$-property has a singleton transversal (alternatively, can be pierced by a single point).

In general, $d+1$ is clearly optimal in Helly's theorem, as any family of $n$ hyperplanes in a general position in $\Re^d$ satisfies the $(d,d)$-property but cannot be pierced by less than $n/d$ points. However, for numerous specific classes of families, a $(d',d')$-property for some $d'<d+1$ is already sufficient to imply piercing by a single point. The minimal number $d'$ for which this holds is called the \emph{Helly number} of the family. For example, any family of {\em axis-parallel boxes} in $\Re^d$ has Helly number 2.

\medskip

In 1957, Hadwiger and Debrunner \cite{HD57} proved the following generalization of Helly's theorem:
\begin{theorem}[Hadwiger-Debrunner Theorem \cite{HD57}]\label{thm:HD-thm}
For all $p \geq q \geq d+1$ such that $q > \frac{d-1}{d}p +1$, any family of compact convex sets in $\Re^d$ that satisfies the $(p,q)$-property can be pierced by $p-q+1$ points.
\end{theorem}

\begin{remark}
The bound in Theorem~\ref{thm:HD-thm} is tight. Indeed, any family of $n$ sets which consists of $p-q$ pairwise disjoint sets and $n-(p-q)$ copies of the same set satisfies the $(p,q)$-property but cannot be pierced by less than $p-q+1$ points.
\end{remark}

Hadwiger and Debrunner conjectured that while for general $p \geq q \geq d+1$, a transversal of size $p-q+1$ is not guaranteed, a $(p,q)$-property does
imply a bounded-size transversal. This conjecture was proved only 35 years later, in the celebrated $(p,q)$-theorem of Alon and Kleitman.
\begin{theorem}[Alon-Kleitman $(p,q)$-Theorem \cite{AK}]
For any triple of positive integers $p \geq q \geq d+1$, there exists an integer $s=s(p,q,d)$ such that if $\F$ is a family of compact convex sets in $\Re^d$ satisfying the $(p,q)$-property, then there exists a transversal for $\F$ of size at most $s$.
\end{theorem}
The smallest value $s$ that works for $p\ge q>d$ is called `the Hadwiger-Debrunner number' and is denoted by $\HD_d(p,q)$. For various specific classes of families, a stronger $(p,q)$-theorem can be obtained. In such cases, we denote the minimal $s$ that works for the family $\F$ by $\HD_{\F}(p,q)$.

The $(p,q)$-theorem has a rich history of variations and generalizations. To mention a few: In 1997, Alon and Kleitman~\cite{AK97} presented a simpler proof of the theorem (that leads to a somewhat weaker quantitative result). Alon et al.~\cite{AKMM} proved in 2001 a `topological' $(p,q)$-theorem for finite families of sets which are a {\em good cover} (i.e., the intersection of every subfamily is either empty or contractible), and B\'{a}r\'{a}ny et al.~\cite{BFMOP14} obtained in 2014 colorful and fractional versions of the theorem.

The size of the transversal guaranteed by the $(p,q)$-theorem is huge, and a large effort was invested in proving better bounds on $\HD_d(p,q)$, both in general and in specific cases. The most recent general result, by the authors and Tardos~\cite{KST17}, shows that for any $\eps>0$, $\HD_d(p,q) \leq p-q+2$ holds for all $(p,q)$ such that $p>p_0(\eps)$ and $q>p^{\frac{d-1}{d}+\eps}$. Yet, no exact values of the Hadwiger-Debrunner number are known except for those given in the Hadwiger-Debrunner theorem. In fact, even the value $\HD_2(4,3)$ is not known, the best bounds being $3 \leq \HD_2(4,3) \leq 13$ (obtained by Kleitman et al.~\cite{KGT01} in 2001).

\subsection{(p,2)-theorems and their applications}

As mentioned above, while no general $(p,q)$-theorems exist for $q \leq d$, such theorems can be proved for various specific families. Especially desirable are $(p,2)$-theorems, which relate the \emph{packing number}, $\nu(\F)$, of the family $\F$ (i.e., the maximum size of a subfamily all of whose members are pairwise disjoint) to its \emph{piercing number}, $\tau(\F)$ (i.e., the minimal size of a piercing set for the family $\F$).

In the last decades, $(p,2)$-theorems were proved for numerous families. In particular, in 1991 K\'{a}rolyi~\cite{Kar91}
proved a $(p,2)$-theorem for axis-parallel boxes in $\Re^d$, guaranteeing piercing by $O(p \log^{d-1}p)$ points. Kim et al.~\cite{KNPS06} proved in 2006 that any family of translates of a fixed convex set in $\Re^d$ that satisfies the $(p,2)$-property can be pierced by $2^{d-1}d^d(p-1)$ points; five years later, Dumitrescu and Jiang~\cite{DJ11} obtained a similar result for homothets of a convex set in $\Re^d$. In 2012, Chan and Har-Peled~\cite{CH12} proved a $(p,2)$-theorem for families of pseudo-discs in the plane, with a piercing number linear in $p$. Two years ago, Govindarajan and Nivasch~\cite{GN15} showed that any family of convex sets in the plane in which among any $p$ sets there is a pair that intersects on a given convex curve $\gamma$, can be pierced by $O(p^8)$ points.

In 2004, Matou\v{s}ek~\cite{MAT04} showed that families of sets with bounded dual VC-dimension have a bounded fractional Helly number. Recently, Pinchasi~\cite{P15} has drawn a similar relation between the union complexity and the fractional Helly number. Each of these results implies a $(p,2)$-theorem for the respective families, using the proof technique of the Alon-Kleitman $(p,q)$-theorem.

Besides their intrinsic interest, $(p,2)$-theorems serve as a tool for obtaining other results. One such result is an \emph{improved Ramsey Theorem}. Consider, for example, a family $\F$ of $n$ axis-parallel rectangles in the plane. The classical Ramsey theorem implies that $\F$ contains a subfamily of size $\Omega(\log n)$, all whose elements are either pairwise disjoint or pairwise intersecting. As was observed by Larman et al.~\cite{PT94}, the aforementioned $(p,2)$-theorem for axis-parallel rectangles~\cite{Kar91} allows obtaining an improved bound of $\Omega(\sqrt{n/\log n})$. Indeed, either $\F$ contains a subfamily of size $\lceil \sqrt{n/\log n} \rceil$ all whose elements are pairwise disjoint, and we are done, or $\F$ satisfies the $(p,2)$-property with $p=\lceil \sqrt{n/\log n} \rceil$. In the latter case, by the $(p,2)$-theorem, $\F$ can be pierced by $O(p \log p)=O(\sqrt{n \log n})$ points. The largest among the subsets of $\F$ pierced by a single point contains at least $\Omega(\frac{n}{\sqrt{n \log n}}) = \Omega(\sqrt{n/\log n})$ rectangles, and all its elements are pairwise intersecting.

Another result that can be obtained from a $(p,2)$-theorem is an improved $(p,q)$-theorem; this will be described in detail below.

\subsection{(p,2)-theorems and (p,q)-theorems for axis-parallel rectangles and boxes}

The $(p,q)$-problem for axis-parallel boxes is almost as old as the general $(p,q)$-problem, and was studied almost as thoroughly (see the survey of Eckhoff~\cite{ECK03}). It was posed in 1960 by Hadwiger and Debrunner~\cite{HD60,HD64}, who proved that any family of axis-parallel rectangles in the plane that satisfies the $(p,q)$-property, for $p \geq q \geq 2$, can be pierced by ${{p-q+2}\choose{2}}$ points. Unlike the $(p,q)$-problem for general families of convex sets, in this problem a finite bound on the piercing number was known from the very beginning, and the research goal has been to improve the bounds on this size, denoted $\HD_{\mathrm{rect}}(p,q)$ for rectangles and $\HD_{\mathrm{d-box}}(p,q)$ for boxes in $\Re^d$.

For rectangles and $q=2$, the quadratic upper bound on $\HD_{\mathrm{rect}}(p,2)$ was improved to $O(p \log p)$ by Wegner (unpublished), and independently, by K\'{a}rolyi~\cite{Kar91}. The best currently known upper bound, which follows from a recursive formula presented by Fon Der Flaass and Kostochka~\cite{FdFK93}, is
\begin{equation}\label{Eq:(p,2)}
\HD_{\mathrm{rect}}(p,2) \leq p \lceil \log_2 p \rceil - 2^{\lceil \log_2 p \rceil} + 1,
\end{equation}
for all $p \geq 2$. On the other hand, it is known that the `optimal possible' answer $p-q+1=p-1$ fails already for $p=4$. Indeed, Wegner~\cite{Weg65} showed that $\HD_{\mathrm{rect}}(4,2)=5$, and by taking $\lceil p/3 \rceil -1$ pairwise disjoint copies of his example, one obtains a family of axis-parallel rectangles that satisfies the $(p,2)$-property but cannot be pierced by less than $\approx 5p/3$ points.

Wegner~\cite{Weg65} conjectured that $\HD_{\mathrm{rect}}(p,2)$ is linear in $p$, and is possibly even bounded by $2p-3$. While Wegner's conjecture is believed to hold (see~\cite{ECK03,GL85}), no improvement of the bound~\eqref{Eq:(p,2)} was found so far.

For rectangles and $q>2$, Hadwiger and Debrunner showed that the exact bound $\HD_{\mathrm{rect}}(p,q)=p-q+1$ holds for all $q \geq p/2+1$. Wegner~\cite{Weg65} and (independently) Dol'nikov~\cite{Dol72} presented recursive formulas that allow leveraging a $(p,2)$-theorem for axis-parallel rectangles into a tight $(p,q)$-theorem. Applying these formulas along with the Hadwiger-Debrunner quadratic upper bound on $\HD_{\mathrm{rect}}(p,2)$, Dol'nikov showed that $\HD_{\mathrm{rect}}(p,q)=p-q+1$ holds for all $2 \leq q \leq p < {{q+1}\choose{2}}$.
Applying the formulas along with the improved bound~\eqref{Eq:(p,2)} on $\HD_{\mathrm{rect}}(p,2)$, Scheller~(\cite{Sch76}, see also~\cite{ECK03}) obtained by a computer-aided computation upper bounds on the minimal $p$ such that $\HD_{\mathrm{rect}}(p,q)=p-q+1$ holds, for all $q \leq 12$. These values suggest that $\HD_{\mathrm{rect}}(p,q)=p-q+1$ holds already for $q=\Omega(\log p)$. However, it appears that the method in which Dol'nikov proved a tight bound in the range $p<{{q+1}\choose{2}}$ does not extend to show a tight bound for all $q=\Omega(\log p)$ (even if~\eqref{Eq:(p,2)} is employed), and in fact, no concrete improvement of Dol'nikov's result was presented (see the survey~\cite{ECK03}).

For axis-parallel boxes in $\Re^d$, the aforementioned recursive formula of~\cite{FdFK93} implies the bound $\HD_{\mathrm{d-box}}(p,2) \leq O(p \log^{d-1}p)$. While it is believed that the correct upper bound is $O(p)$, the result of~\cite{FdFK93} was not improved ever since; the only advancement is a recent result of Chudnovsky et al.~\cite{CSZ17+}, who proved an upper bound of $O(p \log \log p)$ for any family of axis-parallel boxes in which for each two intersecting boxes, a corner of one is contained in the other.

\subsection{Our results}

\medskip \noindent \textbf{From $(p,2)$-theorems to $(p,q)$-theorems}

\medskip \noindent The main result of this paper is a general method for leveraging a $(p,2)$-theorem into a tight $(p,q)$-theorem, applicable to families with Helly number 2. Interestingly, the method does not assume that the sets in $\F$ are convex or compact.
\begin{theorem}\label{Thm:Main}
Let $\F$ be a family of sets in $\Re^d$ such that $\HD_{\F}(2,2)=1$. Assume that for all $2 \leq p \in \mathbb{N}$ we have $\HD_{\F}(p,2) \leq p f(p)$, where $f: [2,\infty) \rightarrow [1,\infty)$ is a differentiable function of $p$ that satisfies $f'(p) \geq \frac{\log_2 e}{p}$ and $\frac{f'(p)}{f(p)} \leq \frac{5}{p}$ for all $p \geq 2$. Denote $T_c(p)=T_c(p,f)=\min \{q: q \geq 2c \cdot f(2p/q)\}$. Then for any $p \geq q \geq 2$ such that $q \geq T_{100}(p)$, we have $\HD_{\F}(p,q)=p-q+1$.
\end{theorem}

While the condition on the function $f(p)$ looks a bit ``scary'', it actually holds for any function $f$ whose growth rate (as expressed by its derivative $f'(p)$ and by the derivative of its logarithm $(\log f(p))'=\frac{f'(p)}{f(p)}$) is between the growth rates of $f(p)=\log_2 p$ and $f(p)=p^5$, including all cases needed in the current paper. The proof can be easily adjusted to work for any $f$ with a polynomial growth rate, at the expense of replacing `100' with a larger constant depending on the degree of the polynomial.

\medskip The first application of our general method is the following theorem for families of axis-parallel rectangles in the plane, obtained using~\eqref{Eq:(p,2)} as the basic $(p,2)$-theorem and some local refinements.
\begin{theorem}\label{Cor:Rect}
$\HD_{\mathrm{rect}}(p,q)=p-q+1$ holds for all $q \geq 7 \log_2 p$.
\end{theorem}

\begin{remark}
Theorem~\ref{Cor:Rect} improves significantly on the best previous result of Wegner (1965) and Dol'nikov (1972), that obtained the exact value $\HD_{\mathrm{rect}}(p,q)=p-q+1$ only for $q > \sqrt{2p}$.
\end{remark}


Another corollary is a tight $(p,q)$-theorem for axis-parallel boxes in $\Re^d$:
\begin{theorem}\label{Cor:Box}
$\HD_{\mathrm{d-box}}(p,q)=p-q+1$ holds for all $q>c \log^{d-1} p$, where $c$ is a universal constant.
\end{theorem}

In the proof of Theorem~\ref{Thm:Main} we deploy the following observation of Wegner and Dol'nikov, which holds for any family $\F$ with Helly number 2:
\begin{equation}\label{Eq:Dol1}
\HD_{\F}(p,q) \leq \HD_{\F}(p-\lambda,q-1) + \lambda-1,
\end{equation}
where $\lambda=\nu(\F)$ is the packing number of $\F$.\footnote{For the sake of completeness, the proof of the observation is presented in Appendix~\ref{app:obs-proof}.} We use an inductive process in which~\eqref{Eq:Dol1} is applied as long as $\F$ contains a sufficiently large pairwise-disjoint set. To treat the case where $\F$ does not contain a `large' pairwise-disjoint set (and thus, $\nu(\F)$ is small), we make use of a combinatorial argument, based on a variant of a `combinatorial dichotomy' presented by the authors and Tardos~\cite{KST17}, which first leverages the $(p,2)$-theorem into a `weak' $(p,q)$-theorem, and then uses that $(p,q)$-theorem to show that if $\nu(\F)$ is `small' then $\tau(\F)<p-q+1$.

%

\medskip \noindent \textbf{From $(2,2)$-theorems to $(p,2)$-theorems}

\medskip \noindent
It is natural to ask, under which conditions a $(2,2)$-theorem implies a $(p,2)$-theorem for all $p>2$.

While in general, a $(2,2)$-theorem does not imply a $(p,2)$-theorem (see an example in Appendix~\ref{app:example}), we prove such an implication for several kinds of families. Our first result here concerns families with Helly number 2.
\begin{theorem}\label{Thm:New-(p,2)1}
Let $\F$ be a family of compact convex sets in $\Re^d$ with Helly number 2. Then $\HD_{\F}(p,2) \leq p^{2d-1}/2^{d-1}$, and consequently, $\HD_{\F}(p,q)=p-q+1$ holds for all $q>c p^{1-\frac{1}{2d-1}}$, where $c=c(d)$ is a constant depending only on the dimension $d$.
\end{theorem}


The second result only assumes the existence of a $(2,2)$-theorem.
\begin{theorem}\label{Thm:New-(p,2)2}
Let $\F$ be a family of compact convex sets in $\Re^d$ that admits a $(2,2)$-theorem. Then:
\begin{enumerate}
\item $\F$ admits a $(p,2)$-theorem for piercing with a bounded number $s=s(p,d)$ of points.

\item If $d=2$, then $\HD_{\F}(p,2) = O(p^8 \log^2 p)$.

\item If $d=2$ and $\F$ has a bounded VC-dimension\footnote{The definition of VC-dimension is recalled in Appendix~\ref{app:(p,2)}.} then $\HD_{\F}(p,2)=O(p^4 \log^2 p)$.
\end{enumerate}
\end{theorem}


Since families with a sub-quadratic union complexity admit a $(2,2)$-theorem and have a bounded VC-dimension, Theorem~\ref{Thm:New-(p,2)2}(3) implies that any family $\F$ of regions in the plane with a sub-quadratic union complexity satisfies $\HD_{\F}(p,2)=O(p^4 \log^2 p)$. This significantly improves over the bound $\HD_{\F}(p,2)=O(p^{16})$ that was obtained for such families in~\cite{KST17}.

%

\subsection{Organization of the paper}

In Section~\ref{sec:rectangles} we demonstrate our general method for leveraging a $(p,2)$-theorem into a tight $(p,q)$-theorem and prove Theorem~\ref{Cor:Rect}. Our new $(p,2)$-theorem for compact convex sets with Helly number 2 (i.e., Theorem~\ref{Thm:New-(p,2)1} above) is presented in Section~\ref{sec:(p,2)}. Finally, the proof of Theorem~\ref{Thm:Main} is presented in Appendix~\ref{app:main}, and the proof of Theorem~\ref{Thm:New-(p,2)2} is presented in Appendix~\ref{app:(p,2)}.

\section{From (p,2)-theorems to tight (p,q)-theorems}
\label{sec:rectangles}

In this section we present our main theorem which allows leveraging a $(p,2)$-theorem into a tight $(p,q)$-theorem, for families $\F$ that satisfy $\HD_{\F}(2,2)=1$. As the proof of the theorem in its full generality is somewhat complex, we present here the proof in the case of axis-parallel rectangles in the plane, and provide the full proof in Appendix~\ref{app:main}. Before presenting the proof of the theorem, we briefly present the Wegner-Dol'nikov argument (parts of which we use in our proof) in Section~\ref{sec:sub:WD}, provide an outline of our method in Section~\ref{sec:sub:outline}, and prove two preparatory lemmas in Section~\ref{sec:sub:lemmas}.

\subsection{The Wegner-Dol'nikov method}
\label{sec:sub:WD}

As mentioned in the introduction, Wegner and (independently) Dol'nikov leveraged the Hadwiger-Debrunner $(p,2)$-theorem for axis-parallel rectangles in the plane, which asserts that $\HD_{\mathrm{rect}}(p,2) \leq {{p}\choose{2}}$, into a tight $(p,q)$-theorem, asserting that $\HD_{\mathrm{rect}}(p,q) \leq p-q+1$ holds for all $p \geq q \geq 2$ such that $p<{{q+1}\choose{2}}$. The heart of the Wegner-Dol'nikov argument is the following observation.\footnote{Observation~\ref{Obs:Dol} is described in Equation~\eqref{Eq:Dol1} above. For the sake of completeness, we present its proof in Appendix~\ref{app:obs-proof}.}
\begin{observation}\label{Obs:Dol}
Let $\F$ be a family that satisfies $\HD_{\F}(2,2)=1$, and put $\lambda=\nu(\F)$. Then
\[
\HD_{\F}(p,q) \leq \HD_{\F}(p-\lambda,q-1) + \lambda-1.
\]
\end{observation}

Using Observation~\ref{Obs:Dol}, Wegner and Dol'nikov proved the following theorem, which we will use in our proof below.
\begin{theorem}[\cite{Dol72}, Theorem~2]\label{Prop:Dol2}
Let $\F$ be a family of axis-parallel rectangles in the plane. Then for any $p \geq q \geq 2$ such that $p<{{q+1}\choose{2}}$, we have
$\HD_{\F}(p,q) =p-q+1$.
\end{theorem}

\begin{proof}
The proof is by induction. The induction basis is $q=2$: for this value, the assertion is relevant only for $p=2$, and we indeed have $\HD_{\mathrm{rect}}(2,2)=1=2-2+1$ as asserted.

For the inductive step, we consider $\lambda=\nu(\F)$. Note that $\F$ satisfies the $(\lambda+1,2)$-property. Thus, if ${{\lambda+1}\choose{2}} \leq p-q+1$ then we have $\HD_{\F}(p,q) \leq p-q+1$ by the aforementioned Hadwiger-Debrunner $(p,2)$-theorem for axis-parallel rectangles. On the other hand, if ${{\lambda+1}\choose{2}} > p-q+1$ then it can be easily checked that $p-\lambda<{{q}\choose{2}}$, so by the induction hypothesis we have $\HD_{\F}(p-\lambda,q-1) = (p-\lambda)-(q-1)+1$. By Observation~\ref{Obs:Dol}, this implies $\HD_{\F}(p,q) \leq \HD_{\F}(p-\lambda,q-1)+\lambda-1=p-q+1$, as asserted.
\end{proof}

\subsection{Outline of our method}
\label{sec:sub:outline}

Let $\F$ be a family of axis-parallel rectangles in the plane. Instead of leveraging the Hadwiger-Debrunner $(p,2)$-theorem for $\F$ into a $(p,q)$-theorem as was done by Wegner and Dol'nikov, we would like to leverage the stronger bound $\HD_{\mathrm{rect}}(p,2)\leq p \log_2 p$ which follows from~\eqref{Eq:(p,2)}. We want to deduce that $\HD_{\mathrm{rect}}(p,q)=p-q+1$ holds for all $q \geq 7\log p$.

Basically, we would like to perform an inductive process similar to the process applied in the proof of Theorem~\ref{Prop:Dol2}. As above, put $\lambda=\nu(\F)$. If $\lambda$ is `sufficiently large' (namely, if $q-1 \geq 7\log_2 (p-\lambda)$), we apply the recursive formula $\HD_{\F}(p,q) \leq \HD_{\F}(p-\lambda,q-1)+\lambda-1$ and use the induction hypothesis to bound $\HD_{\F}(p-\lambda,q-1)$. Otherwise, we would like to use the improved $(p,2)$-theorem to deduce that $\F$ can be pierced by at most $p-q+1$ points.

However, since we want to prove the theorem in the entire range $q \geq 7\log_2 p$, in order to apply the induction hypothesis to  $\HD_{\F}(p-\lambda,q-1)$, $\lambda$ must be at least \emph{linear in $p$} (specifically, we need $\lambda \geq 0.1p$, as is shown below). Thus, in the `otherwise' case we have to show that if $\lambda < 0.1p$, then $\F$ can be pierced by at most $p-q+1$ points. If we merely use the fact that $\F$ satisfies the $(\lambda+1,2)$-property and apply the improved $(p,2)$-theorem, we only obtain that $\F$ can be pierced by $O(p\log p)$ points -- significantly weaker than the desired bound $p-q+1$.

Instead, we use a more complex procedure, partially based on the following observation, presented in~\cite{KST17} (and called there a `combinatorial dichotomy'):
\begin{observation}\label{Obs:our}
Let $\F$ be a family that satisfies the $(p,q)$-property. For any $p' \leq p,q' \leq q$ such that $q' \leq p'$, either $\F$ satisfies the $(p',q')$-property, or there exists $S \subset \F$ of size $p'$ that does not contain an intersecting $q'$-tuple. In the latter case, $\F \setminus S$ satisfies the $(p-p',q-q'+1)$-property.
\end{observation}
First, we use Observation~\ref{Obs:our} to leverage the $(p,2)$-theorem by an inductive process into a `weak' $(p,q)$-theorem that guarantees piercing with $p-q+1+O(p)$ points, for all $q =\Omega(\log p)$. We then show that if $\lambda <0.1p$ then $\F$ can be pierced by at most $p-q+1$ points, by combining the weak $(p,q)$-theorem, another application of Observation~\ref{Obs:our}, and a lemma which exploits the size of $\lambda$.

\subsection{The two main lemmas used in the proof}
\label{sec:sub:lemmas}

Our first lemma leverages the $(p,2)$-theorem $\HD_{\mathrm{rect}}(p,2)\leq p \log_2 p$ into a weak $(p,q)$-theorem, using Observation~\ref{Obs:our}.
\begin{lemma}\label{Lem:Weak-(p,q)}
Let $\F$ be a family of axis-parallel rectangles in the plane. Then for any $c > 0$ and for any $p \geq q \geq 2$ such that $q \geq c \log_2 p$, we have
\[
\HD_{\F}(p,q) \leq p-q+1+\frac{2p}{c}.
\]
\end{lemma}

\begin{proof}
First, assume that both $p$ and $q$ are powers of 2. We perform an inductive process with $\ell = (\log_2 q)-1$ steps, where we set $\F_0=\F$ and $(p_0,q_0)=(p,q)$, and in each step $i$, we apply Observation~\ref{Obs:our} to a family $\F_{i-1}$ that satisfies the $(p_{i-1},q_{i-1})$-property, with $(p',q')=(\frac{p_{i-1}}{2},\frac{q_{i-1}}{2})$ which we denote by $(p_i,q_i)$.

Consider Step $i$. By Observation~\ref{Obs:our}, either $\F_{i-1}$ satisfies the $(p_i,q_i)=(\frac{p_{i-1}}{2},\frac{q_{i-1}}{2})$-property, or there exists a `bad' set $S_i$ of size $\frac{p_{i-1}}{2}$ without an intersecting $\frac{q_{i-1}}{2}$-tuple, and the family $\F_{i-1} \setminus S_i$ satisfies the $(\frac{p_{i-1}}{2},\frac{q_{i-1}}{2}+1)$-property, and in particular, the $(\frac{p_{i-1}}{2},\frac{q_{i-1}}{2})$-property. In either case, we are reduced to a family $\F_i$ (either $\F_{i-1}$ or $\F_{i-1} \setminus S_i$) that satisfies the $(p_i,q_i)$-property, to which we apply Step~$i+1$.

At the end of Step $\ell$ we obtain a family $\F_{\ell}$ that satisfies the $(2p/q,2)$-property. (Note that the ratio between the left term and the right term remains constant along the way.) By the $(p,2)$-theorem, $\F_{\ell}$ can be pierced by $\frac{2p}{q}\log_2 \left(\frac{2p}{q} \right)$ points. As $q \geq \max(c \log_2 p,2)$, this implies that $\F_{\ell}$ can be pierced by
\[
\frac{2p}{q}\log_2 \left(\frac{2p}{q} \right) \leq \frac{2p}{q} \log_2 p \leq \frac{2p}{c}
\]
points.

In order to pierce $\F$, we also have to pierce the `bad' sets $S_i$. In the worst case, in each step we have a bad set, and so we have to pierce $S=\cup_{i=1}^{\ell} S_i$. The size of $S$ is $|S|=\frac{p}{2}+\frac{p}{4}+\ldots+2+1=p-1$. Since any family that satisfies the $(p,q)$-property also satisfies the $(p-k,q-k)$-property for any $k$, the family $S$ contains an intersecting $(q-1)$-tuple, which of course can be pierced by a single point. Hence, $S$ can be pierced by $(p-1)-(q-1)+1=p-q+1$ points. Therefore, in total $\F$ can be pierced by $p-q+1+2p/c$ points, as asserted.

Now, we have to deal with the case where $p,q$ are not necessarily powers of 2, and thus, in some of the steps either $p_{i-1}$ or $q_{i-1}$ or both are not divisible by 2. It is clear from the proof presented above that if we can define $(p_i,q_i)$ in such a way that in both cases (i.e., whether there is a `bad' set or not), we have $\frac{p_i}{q_i} \leq \frac{p_{i-1}}{q_{i-1}}$, and also the total size of the bad sets (i.e., $|S|$) is at most $p$, the assertion can be deduced as above (as the ratio between the left term and the right term only decreases).
We show that this can be achieved by a proper choice of $(p_i,q_i)$ and a slight modification of the steps described above. Let
\[
(p',q') = \left( \lfloor \frac{p_{i-1}}{2} \rfloor, \lceil \frac{q_{i-1}}{2} \rceil \right).
\]
If $\F_{i-1}$ satisfies the $(p',q')$-property, we define $\F_i=\F_{i-1}$ and $(p_i,q_i)=(p',q')$. Otherwise, there exists a `bad' set $S_i$ of size $p'$ that does not contain an intersecting $q'$-tuple, and the family $\F_{i-1} \setminus S_i$ satisfies the
\[
(p_{i-1}-p',q_{i-1}-q'+1) = \left( \lceil \frac{p_{i-1}}{2} \rceil, \lfloor \frac{q_{i-1}}{2} \rfloor +1 \right)
\]
property. In this case, we define $\F_i=\F_{i-1} \setminus S_i$ and $(p_i,q_i)= (p_{i-1}-p',q_{i-1}-q'+1)$.

It is easy to check that in both cases we have $\frac{p_i}{q_i} \leq \frac{p_{i-1}}{q_{i-1}}$, and that $|S| \leq p-1$ holds also with respect to the modified definition of the $S_i$'s. Hence, the proof indeed can be completed, as above.
\end{proof}

Our second lemma is a simple upper bound on the piercing number of a family that satisfies the $(p,2)$-property. We shall use it to show that
if $\nu(\F)$ is `small', then we can save `something' when piercing large subsets of $\F$.
\begin{lemma}\label{Lem:Cute}
Any family $\G$ of $m$ sets that satisfies the $(p,2)$-property can be pierced by $\left \lfloor \frac{m+p-1}{2} \right \rfloor$ points.
\end{lemma}

\begin{proof}
We perform the following simple recursive process. If $\G$ contains a pair of intersecting sets, pierce them by a single point and remove both of them from $\G$. Continue in this fashion until all remaining sets are pairwise disjoint. Then pierce each remaining set by a separate point.

As $\G$ satisfies the $(p,2)$-property, the number of sets that remain in the last step is at most $p-1$ if $m-(p-1)$ is even and at most $p-2$ otherwise. In the former case, the resulting piercing set is of size at most $\frac{m-(p-1)}{2}+(p-1)=\frac{m+p-1}{2}$. In the latter case, the piercing set is of size at most $\frac{m-(p-2)}{2}+(p-2)=\frac{m+p-2}{2}$. Hence, in both cases the piercing set is of size at most $\lfloor \frac{m+p-1}{2} \rfloor$, as asserted.
\end{proof}

\begin{remark}
The assertion of Lemma~\ref{Lem:Cute} is tight, as for a family $\G$ composed of $m-p+2$ lines in a general position in the plane and $p-2$ pairwise-disjoint segments that do not intersect any of the lines, we have $|\G|=m$, $\G$ satisfies the $(p,2)$-property, and $\G$ clearly cannot be pierced by less than $\left \lfloor \frac{m+p-1}{2} \right \rfloor$ points.
\end{remark}

\begin{corollary}
Let $\F$ be a family of sets in $\Re^d$, and put $\lambda=\nu(\F)$. Then any subset $S \subset \F$ can be pierced by at most $\left \lfloor \frac{|S|+\lambda}{2} \right \rfloor $ points.
\end{corollary}
The corollary follows from the lemma immediately, as any such family $\F$ satisfies the $(\lambda+1,2)$-property.

\subsection{Proof of Theorem~\ref{Cor:Rect}}
\label{sec:sub:proof}

Now we are ready to present the proof of our main theorem, in the specific case of axis-parallel rectangles in the plane. Let us recall its statement.

\medskip \noindent \textbf{Theorem~\ref{Cor:Rect}.} Let $\F$ be a family of axis-parallel rectangles in the plane. If $\F$ satisfies the $(p,q)$-property, for $p \geq q \geq 2$ such that $q \geq 7 \log_2 p$, then $\F$ can be pierced by $p-q+1$ points.

\begin{remark}
We note that the parameters in the proof (e.g., the values of $(p',q')$ in the inductive step) were chosen in a sub-optimal way, that is however sufficient to yield the assertion with the constant $7$. (The straightforward choice $(p',q')=(0.5p,0.5q)$ is not sufficient for that). The constant can be further optimized by a more careful choice of the parameters; however, it seems that in order to reduce it below 6, a significant change in the proof is needed.
\end{remark}

\begin{proof}[Proof of Theorem~\ref{Cor:Rect}] The proof is by induction.

\medskip \noindent \textbf{Induction basis.} One can assume that $q \geq 37$, as for any smaller value of $q$, there are no $p$'s such that $7 \log_2 p \leq q \leq p$. For $q=37$, the theorem is only relevant for $(p,q)=(37,37)$, and in this case we clearly have $\HD_{\F}(p,q)=1=p-q+1$. In fact, this is a sufficient basis, since, in the inductive step, the value of $q$ is reduced by 1 every time. However, in the proof we would like to assume that $p,q$ are `sufficiently large'; hence, we use Theorem~\ref{Prop:Dol2} as the induction basis in order to cover a larger range of small $(p,q)$ values.

We observe that for $q \leq 70$, all relevant $(p,q)$ pairs (i.e., all pairs for which $7 \log_2 p \leq q \leq p$) satisfy $p \leq {{q+1}\choose{2}}$. Hence, in this range we have $\HD_{\F}(p,q)=p-q+1$ by Theorem~\ref{Prop:Dol2}. Therefore, we may assume that $q >70$; we also may assume $q<\sqrt{2p}$ (as otherwise, the assertion follows from Theorem~\ref{Prop:Dol2}), and thus, $p>2450$ and so (using again the assumption $q<\sqrt{2p}$), also $p>35q$.

\medskip \noindent \textbf{Inductive step.} Put $\lambda=\nu(\F)$. By Observation~\ref{Obs:Dol}, we have $\HD_{\F}(p,q) \leq \HD_{\F}(p-\lambda,q-1)+\lambda-1$. We want $\lambda$ to be sufficiently large, such that if $(p,q)$ lies in the range covered by the theorem (i.e., if $q \geq 7 \log_2 p$), then $(p-\lambda,q-1)$ also lies in the range covered by the theorem (i.e., $q-1 \geq 7\log_2 (p-\lambda)$). Note that the condition $q \geq 7\log_2 p$ is equivalent to $2^{q/7} \geq p$, which implies $2^{(q-1)/7} = \frac{2^{q/7}}{2^{1/7}} \geq 0.9p$. Hence, if $\lambda \geq 0.1p$ then $q-1 \geq 7\log_2(p-\lambda)$, and so we can deduce from the induction hypothesis that
\[
\HD_{\F}(p,q) \leq \HD_{\F}(p-\lambda,q-1)+\lambda-1 \leq (p-\lambda)-(q-1)+1+(\lambda-1)=p-q+1,
\]
as asserted. Therefore, it is sufficient to prove that $\HD_{\F}(p,q) \leq p-q+1$ holds when $\lambda<0.1p$.

\medskip Under this assumption on $\lambda$, we apply Observation~\ref{Obs:our} to $\F$, with $(p',q')=\left(\lfloor 0.62p \rfloor, 0.5q \right)$. We have to consider two cases:

\medskip \noindent \textbf{Case 1: $\F$ satisfies the $(p',q')$-property.} By the assumption on $(p,q)$, we have $q \geq 7\log_2 p$, and thus, $0.5q \geq 3.5\log_2 p \geq 3.5 \log_2 \lfloor 0.62p \rfloor$. Hence, by Lemma~\ref{Lem:Weak-(p,q)},
\[
\HD_{\F}\left(\lfloor 0.62p \rfloor, 0.5q \right) \leq 0.62p-0.5q+1+\frac{2}{3.5}\cdot 0.62p<0.975p-0.5q+1 \leq p-q+1,
\]
where the last inequality holds because we may assume $q \leq 0.05p$, since $p>35q$ as was written above. Thus, $\F$ can be pierced by at most $p-q+1$ points, as asserted.

\medskip \noindent \textbf{Case 2: $\F$ does not satisfy the $(p',q')$-property.} In this case, there exists a `bad' subfamily $S$ of size $p'=\lfloor 0.62p \rfloor$ that does not contain an intersecting $0.5q$-tuple, and the family $\F \setminus S$ satisfies the $(\lceil 0.38p \rceil,0.5q)$-property.

To pierce $\F \setminus S$, we use Lemma~\ref{Lem:Weak-(p,q)}. Like above, we have $0.5q \geq 3.5\log_2 \lceil 0.38p \rceil$, whence by Lemma~\ref{Lem:Weak-(p,q)},
\[
\HD_{\F}\left(\lceil 0.38p \rceil, 0.5q \right) \leq 0.39p-0.5q+1+\frac{2}{3.5}\cdot 0.39p < 0.613p-0.5q+1,
\]
where the first inequality holds since we may assume $p \geq 100$ (as was written above), and thus, $\lceil 0.38p \rceil \leq 0.39p$.

To pierce the `bad' subfamily $S$, we use Lemma~\ref{Lem:Cute}, which implies that $S$ can be pierced by
\[
\lfloor \frac{1}{2}(|S|+\lambda) \rfloor \leq \frac{1}{2}(0.62p+0.1p)= 0.36p
\]
points. Therefore, in total $\F$ can be pierced by $(0.613p-0.5q+1)+0.36p<0.975p-0.5q+1$ points. Since we may assume $q \geq 0.05p$ (like above), this implies that $\F$ can be pierced by $p-q+1$ points. This completes the proof.
\end{proof}

\section{From (2,2)-theorems to (p,2)-theorems}
\label{sec:(p,2)}

As was mentioned in the introduction, in general, the existence of a $(2,2)$-theorem (and even Helly number 2) does not imply the existence of a $(p,2)$-theorem. An example mentioned by Fon der Flaass and Kostochka~\cite{FdFK93} (in a slightly different context) is presented in Appendix~\ref{app:example}.

\medskip \noindent In this section we prove Theorem~\ref{Thm:New-(p,2)1} which asserts that for compact convex families with Helly number~2, a $(2,2)$-theorem does imply a $(p,2)$-theorem, and consequently, a tight $(p,q)$-theorem for a large range of $q$'s. Due to space constraints, the proof of our other new $(p,2)$-theorem (i.e., Theorem~\ref{Thm:New-(p,2)2}) is presented in Appendix~\ref{app:(p,2)}.


Let us recall the assertion of the theorem:

\medskip \noindent \textbf{Theorem~\ref{Thm:New-(p,2)1}.} For any family $\F$ of compact convex sets in $\Re^d$ that has Helly number 2, we have $\HD_{\F}(p,2) \leq \frac{p^{2d-1}}{2^{d-1}}$. Consequently, we have $\HD_{\F}(p,q)=p-q+1$ for all $q>c p^{1-\frac{1}{2d-1}}$, where $c=c(d)$.

\medskip \noindent The `consequently' part follows immediately from the $(p,2)$-theorem via Theorem~\ref{Thm:Main}. (Formally, Theorem~\ref{Thm:Main} is stated only for growth rate of  $\HD_{\F}(p,2)=O(p^5)$, but it is apparent from the proof that the argument can be extended to $\HD_{\F}(p,2)=O(p^m)$ for any $m \in \mathbb{N}$, at the expense of the constant $c$ becoming dependent on $m$.) Hence, we only have to prove the $(p,2)$-theorem.

\medskip Let us present the proof idea first. The proof goes by induction on $d$. Given a family $\F$ of sets in $\Re^d$ that satisfies the assumptions of the theorem and has the $(p,2)$-property, we take $\mathcal{S}$ to be a maximum (with respect to size) pairwise-disjoint subfamily of $\F$, and consider the intersections of other sets of $\F$ with the elements of $\mathcal{S}$. We observe that by the maximality of $\mathcal{S}$, each set $A \in \F \setminus \mathcal{S}$ intersects at least one element of $\mathcal{S}$, and thus, we may partition $\F$ into three subfamilies: $\mathcal{S}$ itself, the family $\mathcal{U}$ of sets in $\F \setminus \mathcal{S}$ that intersect only one element of $\mathcal{S}$, and the family $\mathcal{M} \subset \F \setminus \mathcal{S}$ of sets that intersect at least two elements of $\mathcal{S}$.

We show (using the maximality of $\mathcal{S}$ and the $(2,2)$-theorem on $\F$) that $\mathcal{U} \cup \mathcal{S}$ can be pierced by $p-1$ points. As for $\mathcal{M}$, we represent it as a union of families: $\mathcal{M}= \cup_{C,C' \in \mathcal{S}} \mathcal{X}_{C,C'}$, where each $\mathcal{X}_{C,C'}$ consists of the elements of $F \setminus \mathcal{S}$ that intersect both $C$ and $C'$. We use a geometric argument to show that each $\mathcal{X}_{C,C'}$ corresponds to $\mathcal{Y}_{C,C'} \subset \Re^{d-1}$ that has Helly number 2 and satisfies the $(p,2)$-property. This allows us to bound the piercing number of $\mathcal{Y}_{C,C'}$ by the induction hypothesis, and consequently, to bound the piercing number of $\mathcal{X}_{C,C'}$. Adding up the piercing numbers of all $\mathcal{X}_{C,C'}$'s and the piercing number of $\mathcal{U} \cup \mathcal{S}$ completes the inductive step.


\begin{proof}[Proof of Theorem~\ref{Thm:New-(p,2)1}] By induction on $d$.

\medskip \noindent \textbf{Induction basis.} For any family $\F$ of compact convex sets in $\Re^1$, by the Hadwiger-Debrunner theorem~\cite{HD57} we have $\HD_{\F}(p,2)=p-2+1<p=p^{2 \cdot 1-1}/2^{1-1}$, and so the assertion holds.

\medskip \noindent \textbf{Inductive step.} Let $\F$ be a family of sets in $\Re^d$ that satisfies the assumptions of the theorem and has the $(p,2)$-property. Let $\mathcal{S}$ be a maximum (with respect to size) pairwise-disjoint subfamily of $\F$. W.l.o.g., we may assume $|\mathcal{S}|=p-1$.

By the maximality of $\mathcal{S}$, each set $A \in \F \setminus \mathcal{S}$ intersects at least one element of $\mathcal{S}$. Moreover, any two sets $A,B \in \F$ that intersect the same $C \in \mathcal{S}$ and do not intersect any other element of $\mathcal{S}$, are intersecting, as otherwise, the subfamily $\mathcal{S} \cup \{A,B\} \setminus \{C\}$ would be a pairwise-disjoint subfamily of $\F$ that is larger than $\mathcal{S}$, a contradiction. Hence, for each $C_0 \in \mathcal{S}$, the subfamily
\[
\mathcal{X}_{C_0} = \{A \in \F: \{C \in \mathcal{S}: A \cap C \neq \emptyset\}=\{C_0\}\} \cup \{C_0\}
\]
satisfies the $(2,2)$-property, and thus, can be pierced by a single point by the assumption on $\F$. Therefore, denoting $\mathcal{U}= \{A \in \F: |\{C \in \mathcal{S}: A \cap C \neq \emptyset\}|=1\}$, all sets in $\mathcal{U} \cup \mathcal{S}$ can be pierced by at most $p-1$ points.

Let $\mathcal{M} \subset \F$ be the family of all sets in $\F$ that intersect at least two elements of $\mathcal{S}$. For each $C,C' \in \mathcal{S}$, let
\[
\mathcal{X}_{C,C'} = \{A \in \F \setminus \mathcal{S}: A \cap C \neq \emptyset \wedge A \cap C' \neq \emptyset\}.
\]
(Note that the elements of $\mathcal{X}_{C,C'}$ may intersect other elements of $\mathcal{S}$). Let $H \subset \Re^d$ be a hyperplane that strictly separates $C$ from $C'$, and put $\mathcal{Y}_{C,C'} = \{A \cap H: A \in \mathcal{X}_{C,C'}\}$.
\begin{claim}\label{Cl:Aux-g}
$\mathcal{Y}_{C,C'} \subset H \approx \Re^{d-1}$ admits $\HD_{\mathcal{Y}_{C,C'}}(2,2)=1$ and satisfies the $(p,2)$-property.
\end{claim}

\begin{proof}
To prove the claim, we observe that $A \cap H, A' \cap H \in \mathcal{Y}_{C,C'}$ intersect if and only if $A$ and $A'$ intersect. Indeed, assume $A \cap A' \neq \emptyset$. The family $\{A,A',C\}$ satisfies the $(2,2)$-property, and hence, can be pierced by a single point by the assumption on $\F$. Thus, $A \cap A'$ contains a point of $C$. For the same reason, $A \cap A'$ contains a point of $C'$. Therefore, $A \cap A'$ contains points on the two sides of the hyperplane $H$. However, $A \cap A'$ is convex, and so, $(A \cap A') \cap H \neq \emptyset$, which means that $(A \cap H)$ and $(A' \cap H)$ intersect. The other direction is obvious.

It is now clear that as $\mathcal{X}_{C,C'} \subset \F$ satisfies the $(p,2)$-property, $\mathcal{Y}_{C,C'}$ satisfies the $(p,2)$-property as well. Moreover, let $T = \{A_1 \cap H, A_2 \cap H, A_3 \cap H, \ldots \} \subset \mathcal{Y}_{C,C'}$ be pairwise-intersecting. The corresponding family $\tilde{T} = \{C,A_1,A_2,A_3,\ldots\}$ is pairwise-intersecting, and thus, can be pierced by a single point by the assumption on $\F$. Thus, $(A_1 \cap A_2 \cap A_3 \cap \ldots) \cap C \neq \emptyset$. For the same reason, $(A_1 \cap A_2 \cap A_3 \cap \ldots) \cap C' \neq \emptyset$. Since $A_1 \cap A_2 \cap A_3 \cap \ldots$ is convex, this implies that $(A_1 \cap A_2 \cap A_3 \cap \ldots) \cap H \neq \emptyset$, or equivalently, that the family $T$ can be pierced by a single point. Therefore, $\mathcal{Y}_{C,C'}$ satisfies $\HD_{\mathcal{Y}_{C,C'}}(2,2)=1$, as asserted.
\end{proof}

\noindent Claim~\ref{Cl:Aux-g} allows us to apply the induction hypothesis to $\mathcal{Y}_{C,C'}$, to deduce that it can be pierced by less than $p^{2d-3}/2^{d-1}$ points. Since $\mathcal{S}$ contains only ${{p-1}\choose{2}}$ pairs $(C,C')$, and since any set in $\mathcal{M}$ belongs to at least one of the $\mathcal{X}_{C,C'}$, this implies that $\mathcal{M}$ can be pierced by less than ${{p-1}\choose{2}}  \cdot p^{2d-3}/2^{d-2}$ points. As $\mathcal{U} \cup \mathcal{S}$ can be pierced by $p-1$ points as shown above, $\F$ can be pierced by less than
\[
{{p-1}\choose{2}}  \cdot \frac{p^{2d-3}}{2^{d-2}} + (p-1) < \frac{p^{2d-1}}{2^{d-1}}
\]
points. This completes the proof.
\end{proof}

\bibliographystyle{plain}
\bibliography{references1}

\appendix

\section{Proof of Theorem~\ref{Thm:Main}}
\label{app:main}

In this appendix we present the full proof of Theorem~\ref{Thm:Main}, which allows leveraging a $(p,2)$-theorem into a tight $(p,q)$-theorem, for families $\F$ that satisfy $\HD_{\F}(2,2)=1$. For the sake of completeness we present the proof almost in full, although most of the components appear (in a simplified form) in the case of axis-parallel rectangles presented in Section~\ref{sec:rectangles}. This appendix is organized as follows. First we outline the proof in Section~\ref{app:sub:outline}. Then we present several lemmas required for the proof in Section~\ref{app:sub:lemmas}, and the proof itself in Section~\ref{app:sub:proof}. We deduce Theorem~\ref{Cor:Box} from Theorem~\ref{Thm:Main} in Section~\ref{app:sub:box}. Finally, for sake of completeness we present the proof of Observation~\ref{Obs:Dol} in Section~\ref{app:obs-proof}.

\subsection{Outline of our method}
\label{app:sub:outline}

Let $\F$ be a family that satisfies $\HD_{\F}(2,2)=1$. In order to leverage a $(p,2)$-theorem for $\F$ into a tight $(p,q)$-theorem we would like to perform an inductive process similar to the process applied in the proof of Theorem~\ref{Prop:Dol2}. Put $\lambda=\nu(\F)$. If $\lambda$ is `sufficiently large', we apply the recursive formula $\HD_{\F}(p,q) \leq \HD_{\F}(p-\lambda,q-1)+\lambda-1$ and use the induction hypothesis to bound $\HD_{\F}(p-\lambda,q-1)$. Otherwise, we would like to use the $(p,2)$-theorem to deduce that $\F$ can be pierced by at most $p-q+1$ points.

Since we allow $q$ to be as small as roughly $\log p$, and as we want to apply the induction hypothesis to $\HD_{\F}(p-\lambda,q-1)$, $\lambda$ must be at least \emph{linear in $p$}. Thus, in the `otherwise' case we have to show directly that if $\lambda <c'p$ for a sufficiently small constant $c'$, then $\F$ can be pierced by at most $p-q+1$ points. If we merely use the fact that $\F$ satisfies the $(\lambda+1,2)$-property and apply the $(p,2)$-theorem, we only obtain that $\F$ can be pierced by $c'p f(c'p)$ points -- significantly weaker than the desired bound $p-q+1$.

Instead, we use a more complex procedure, based on Observation~\ref{Obs:our} presented above. First, we use Observation~\ref{Obs:our} to leverage the $(p,2)$-theorem by an inductive process into a `weak' $(p,q)$-theorem that guarantees piercing with $p-q+1+O(p)$ points, for all $q =\Omega(T_{100}(p))$, where $T_c(p)=\min \{q: q \geq 2c \cdot f(2p/q)\}$. We then show that if $\lambda <c'p$ for a sufficiently small absolute constant $c'$, then $\F$ can be pierced by at most $p-q+1$ points, by combining the weak $(p,q)$-theorem, another application of Observation~\ref{Obs:our}, and a lemma which exploits the size of $\lambda$.

In addition, we have to handle the induction basis: while in the proof of Theorem~\ref{Prop:Dol2}, Dol'nikov could use the case $p=q=2$ as the induction basis, our assertion applies only to significantly larger values of $q$. Hence, we will have to guarantee that for the `minimum relevant' value of $q$, for all `relevant' values of $p$ (i.e., all values of $p$ such that $q \geq T_{100}(p)$) we have $\HD_{\F}(p,q)=p-q+1$. We shall deduce this from another result of Dol'nikov presented below.

\subsection{Lemmas used in the proof}
\label{app:sub:lemmas}

The first lemma is a weak $(p,q)$-theorem, that can be obtained from a $(p,2)$-theorem using Observation~\ref{Obs:our}. While the proof of the lemma is very similar to the proof that was already described before, in the case of axis-parallel rectangles, we present it in full for sake of completeness.
\begin{lemma}\label{Lem:App-Weak-(p,q)}
Let $\F$ be a family of sets in $\Re^d$ and let $c >0$. Assume that for all $2 \leq p \in \mathbb{N}$ we have $\HD_{\F}(p,2)=p f(p)$, where $f:[2,\infty) \rightarrow [1,\infty)$ is a monotone increasing function of $p$. Let $T_c(p)=\min \{q: q \geq 2c \cdot f(2p/q)\}$. Then for any $q \geq T_c(p)$, we have
\[
\HD_{\F}(p,q) \leq p-q+1+\frac{p}{c}.
\]
\end{lemma}

\begin{proof}
First, assume that both $p$ and $q$ are powers of 2. We perform an inductive process with $\ell = (\log_2 q)-1$ steps, where we set $\F_0=\F$ and $(p_0,q_0)=(p,q)$, and in each step $i$, we apply Observation~\ref{Obs:our} to a family $\F_{i-1}$ that satisfies the $(p_{i-1},q_{i-1})$-property, with $(p',q')=(\frac{p_{i-1}}{2},\frac{q_{i-1}}{2})$ which we denote by $(p_i,q_i)$.

Consider Step $i$. By Observation~\ref{Obs:our}, either $\F_{i-1}$ satisfies the $(p_i,q_i)=(\frac{p_{i-1}}{2},\frac{q_{i-1}}{2})$-property, or there exists a `bad' set $S_i$ of size $\frac{p_{i-1}}{2}$ without an intersecting $\frac{q_{i-1}}{2}$-tuple, and the family $\F_{i-1} \setminus S_i$ satisfies the $(\frac{p_{i-1}}{2},\frac{q_{i-1}}{2}+1)$-property, and in particular, the $(\frac{p_{i-1}}{2},\frac{q_{i-1}}{2})$-property. In either case, we are reduced to a family $\F_i$ (either $\F_{i-1}$ or $\F_{i-1} \setminus S_i$) that satisfies the $(p_i,q_i)$-property, to which we apply Step~$i+1$.

At the end of Step $\ell$ we obtain a family $\F_{\ell}$ that satisfies the $(2p/q,2)$-property. By the assumption of the lemma, this family can be pierced by $\frac{2p}{q} f (\frac{2p}{q})$ points. Noting that the map $q \mapsto f(2p/q)$ is decreasing and using the definition of $T_c(p)$ and the assumption $q>T_c(p)$, we obtain
\[
\frac{2p}{q} f \left(\frac{2p}{q} \right) \leq \frac{2p}{T_c(p)} f \left(\frac{2p}{T_c(p)} \right) \leq \frac{2p}{2c}=\frac{p}{c},
\]
and thus, $\F_{\ell}$ can be pierced by $p/c$ points.

In order to pierce $\F$, we also have to pierce the `bad' sets $S_i$. In the worst case, in each step we have a bad set, and so we have to pierce $S=\cup_{i=1}^{\ell} S_i$. The size of $S$ is $|S|=\frac{p}{2}+\frac{p}{4}+\ldots+2+1=p-1$. Since any family that satisfies the $(p,q)$-property also satisfies the $(p-k,q-k)$-property for any $k$, the family $S$ contains an intersecting $(q-1)$-tuple, which of course can be pierced by a single point. Hence, $S$ can be pierced by $(p-1)-(q-1)+1=p-q+1$ points. Therefore, in total $\F$ can be pierced by $p-q+1+p/c$ points, as asserted.

Now, we have to deal with the case where $p,q$ are not necessarily powers of 2, and thus, in some of the steps either $p_{i-1}$ or $q_{i-1}$ or both are not divisible by 2. It is clear from the proof presented above that if we can define $(p_i,q_i)$ in such a way that in both cases (i.e., whether the $(p_i,q_i)$-property is satisfied or not), we have $\frac{p_i}{q_i} \leq \frac{p_{i-1}}{q_{i-1}}$, and also the total size of the bad sets (i.e., $|S|$) is at most $p$, the assertion can be deduced as above.
We show that this can be achieved by a proper choice of $(p_i,q_i)$ and a slight modification of the steps described above. Let
\[
(p',q') = \left( \lfloor \frac{p_{i-1}}{2} \rfloor, \lceil \frac{q_{i-1}}{2} \rceil \right).
\]
If $\F_{i-1}$ satisfies the $(p',q')$-property, we define $\F_i=\F_{i-1}$ and $(p_i,q_i)=(p',q')$. Otherwise, there exists a `bad' set $S_i$ of size $p'$ that does not contain an intersecting $q'$-tuple, and the family $\F_{i-1} \setminus S_i$ satisfies the
\[
(p_{i-1}-p',q_{i-1}-q'+1) = \left( \lceil \frac{p_{i-1}}{2} \rceil, \lfloor \frac{q_{i-1}}{2} \rfloor +1 \right)
\]
property. In this case, we define $\F_i=\F_{i-1} \setminus S_i$ and $(p_i,q_i)= (p_{i-1}-p',q_{i-1}-q'+1)$.

It is easy to check that in both cases we have $\frac{p_i}{q_i} \leq \frac{p_{i-1}}{q_{i-1}}$, and that $|S| \leq p-1$ holds also with respect to the modified definition of the $S_i$'s. Hence, the proof indeed can be completed, as above.
\end{proof}

We also need the following easy extension of the classical Hadwiger-Debrunner theorem, obtained by Dol'nikov~\cite{Dol72}.
\begin{proposition}[\cite{Dol72}, Theorem~1]\label{Prop:Dol1}
Let $\F$ be a family that satisfies $\HD_{\F}(2,2)=1$. Then for any $p \geq q \geq 2$ such that $p \leq 2q-2$, we have $\HD_{\F}(p,q) =p-q+1$.
\end{proposition}


The third lemma we use concerns several properties of the function $T_c(p)$.
\begin{claim}\label{Cl:T(p)}
Let $f:[2,\infty) \rightarrow [1,\infty)$ be an increasing function of $p$, let $c>0$, and let $T_c(p)=\min \{q: q \geq 2c \cdot f(2p/q)\}$. Then:
\begin{enumerate}
\item For each $p$, the condition $q \geq 2c \cdot f(2p/q)$ holds for all $T_c(p) \leq q \leq p$.

\item $T_c(p)$ is a non-decreasing function of $p$.


\item If, in addition, $f$ satisfies $f'(p) \geq \frac{\log_2 e}{p}$ for all $p \geq 1$, then:
\begin{enumerate}
\item For all $k \in \mathbb{N}$, all $c \geq \frac{1}{2}\log_{2(k-1)/k} 2$ and all $p$ such that $k \leq T_c(p-1) \leq p-1$, we have $T_c(2p-1) \geq T_c(p-1)+1$.

\item For all $k \in \mathbb{N}$, all $0 < \alpha < \frac{k-1}{k}$, all $c \geq \frac{1}{2}\log_{(k-1)/(\alpha k) } 2$ and all $p$ such that $k \leq T_c(\alpha p) \leq \alpha p$, we have $T_c(p) \geq T_c(\alpha p)+1$.
\end{enumerate}
\end{enumerate}
\end{claim}

\begin{proof}
Properties~(1),(2) follow immediately from the definition of $T_c(p)$ and the assumption that $f$ is increasing. To prove~(3a), consider some specific value of $p$ and denote $T_c(p-1)=q_0$. By the definition of $T_c(p-1)$, we have $q_0-1<2c \cdot f(\frac{2(p-1)}{q_0-1})$. We want to show that $q_0<2c \cdot f(\frac{2(2p-1)}{q_0})$ (which will imply $T_c(2p-1) \geq q_0+1$ by Property~(1)). It is clearly sufficient to show that for any $p$ for which $q_0=T_c(p-1) \geq k$, we have
\[
f \left(\frac{2(2p-1)}{q_0} \right) - f \left(\frac{2(p-1)}{q_0-1} \right) \geq \frac{1}{2c}.
\]
By the assumption on the derivative $f'$, for any $t>t'$ we have
\[
f(t)-f(t')= \int_{t'}^t f'(x)dx \geq \int_{t'}^t \frac{\log_2 e}{x} dx = \log_2 \left(\frac{t}{t'} \right).
\]
Hence,
\begin{align*}
f \left(\frac{2(2p-1)}{q_0}\right) - f \left(\frac{2(p-1)}{q_0-1}\right) &\geq \log_2 \left(\frac{2(2p-1)(q_0-1)}{2q_0(p-1)} \right) \\
&\geq \log_2 \left( \frac{2(q_0-1)}{q_0} \right) \geq \log_2 \left( \frac{2(k-1)}{k} \right),
\end{align*}
where the last inequality holds since $q_0 \geq k$ by assumption. Since $\frac{1}{2c} \leq \log_2 ( \frac{2(k-1)}{k})$ by the assumption on $c$, this completes the proof of~(3a). The proof of~(3b) is almost identical to the proof of~(3a), with a general $\alpha$ instead of $1/2$, and thus is omitted.
\end{proof}

\medskip \noindent As was outlined in Section~\ref{app:sub:outline}, in order to apply the inductive step of the proof, we have to assume that $\lambda=\nu(\F)$ is `sufficiently large'. Specifically, in the inductive step we move from a $(p',q')$-property to a $(p'-\lambda,q'-1)$-property. We assume that $(p',q')$ lies in the range covered by the theorem, i.e., that $q' \geq T_c(p')$, and want to deduce that $(p'-\lambda,q'-1)$ also lies in the range covered by the theorem, i.e., that $q'-1 \geq T_c(p'-\lambda)$. It is clearly sufficient to take $\lambda=\lambda(p')$ such that
\begin{equation}\label{Eq:Lambda1}
T_c(p'-\lambda) \leq T_c(p')-1.
\end{equation}
Our fourth lemma states how large should $\lambda$ be in order to guarantee this, for the particular choice $c=100$ that we use in Theorem~\ref{Thm:Main}.
\begin{lemma}\label{Lem:Lambda}
Assume that $(p',q')$ lies in the range covered by Theorem~\ref{Thm:Main}, i.e., that $q' \geq T_{100}(p')$. Then $q'-1 \geq T_{100}(0.99p')$, and thus, $(0.99p',q'-1)$ lies in the range covered by Theorem~\ref{Thm:Main} as well.
\end{lemma}

\begin{proof}
Denoting $\lambda(p')=\beta p'$, we can apply Claim~\ref{Cl:T(p)}(3b) with $\alpha=1-\beta$ to deduce that~\eqref{Eq:Lambda1} holds for all $c \geq \frac{1}{2}\log_{(k-1)/((1-\beta) k) } 2$, provided that $T_c((1-\beta)p) \geq k$ and $1-\beta<\frac{k-1}{k}$.

In the special case $c=100$, we may take $k=200$ (as $T_{100}(p) = \min \{q: q \geq 200 f(2p/q)\} \geq 200 f(2) \geq 200$ for any $p$). Hence, we may take $\alpha$ to be any number in $(0,\frac{199}{200})$ such that $100 \geq \frac{1}{2}\log_{199/200(1-\beta)} 2$. In particular, $\alpha = 0.99$ works.
\end{proof}

\subsection{Proof of Theorem~\ref{Thm:Main}}
\label{app:sub:proof}

We are ready to prove our main theorem. Let us recall its statement.

\medskip \noindent \textbf{Theorem~\ref{Thm:Main}.} Let $\F$ be a family of sets in $\Re^d$ such that $\HD_{\F}(2,2)=1$. Assume that for all $2 \leq p \in \mathbb{N}$ we have $\HD_{\F}(p,2) \leq p f(p)$, where $f: [2,\infty) \rightarrow [1,\infty)$ is a differentiable function of $p$ that satisfies $f'(p) \geq \frac{\log_2 e}{p}$ and $\frac{f'(p)}{f(p)} \leq \frac{5}{x}$ for all $p \geq 2$. Denote $T_c(p)=\min \{q: q \geq 2c \cdot f(2p/q)\}$. Then for any $p \geq q \geq 2$ such that $q \geq T_{100}(p)$, we have $\HD_{\F}(p,q)=p-q+1$.

\begin{proof}
By induction. We start with the inductive step, and leave the induction basis for the end.

Put $\lambda=\nu(\F)$. By Observation~\ref{Obs:Dol}, we have $\HD_{\F}(p,q) \leq \HD_{\F}(p-\lambda,q-1)+\lambda-1$. If $\lambda \geq 0.01p$, then by Lemma~\ref{Lem:Lambda}, the pair $(p-\lambda,q-1)$ satisfies the assumption of the theorem, and thus, by the induction hypothesis we have $\HD_{\F}(p-\lambda,q-1)=(p-\lambda)-(q-1)+1$, whence $\HD_{\F}(p,q)=p-q+1$ as asserted. Therefore, it is sufficient to prove that if $\lambda<0.01p$, then $\F$ can be pierced by at most $p-q+1$ points.

We apply Observation~\ref{Obs:our} to $\F$, with $(p',q')=\left(\frac{2}{3}p,\frac{q}{2}\right)$. (For the sake of simplicity, we assume that $p,q$ are divisible by $3$ and $2$, respectively. It will be apparent that this does not affect the proof). We have to consider two cases:

\medskip \noindent \textbf{Case 1: $\F$ satisfies the $(p',q')$-property.} Note that by the assumption on $\frac{f'(p)}{f(p)}$, for any $p \geq 2$ we have
\[
\ln( f(4p/3))-\ln(f(p))=\int_{p}^{4p/3} [\ln f(t)]' dt = \int_{p}^{4p/3} \frac{f'(t)}{f(t)} dt \leq \int_{p}^{4p/3} \frac{5}{t} dt = 5 \ln(4/3) < \ln 5,
\]
and thus, $f(4p/3) < 5f(p)$. By the assumption on $(p,q)$, we have $q \geq 200f(2p/q)$, and thus,
\[
\frac{q}{2} \geq 100f \left(\frac{2p}{q} \right) > 20f \left(\frac{4}{3} \frac{2p}{q} \right) = 20f \left( \frac{2 \cdot 2p/3}{q/2} \right).
\]
By the definition of $T_c$, this implies $\frac{q}{2} \geq T_{10}(\frac{2p}{3})$. Therefore, we can apply Lemma~\ref{Lem:App-Weak-(p,q)} to deduce
\[
\HD_{\F}(p',q')=\HD_{\F}\left(\frac{2}{3}p,\frac{q}{2}\right) \leq \frac{2}{3}p-\frac{1}{2}q+1+\frac{1}{10}\cdot \frac{2}{3}p<0.74p-0.5q+1.
\]
This shows that $\F$ can be pierced by less than $p-q+1$ points, if we may assume $0.26p \geq 0.5q$, or equivalently, $q \leq 0.52p$.

To see that we indeed may assume this, note that by Proposition~\ref{Prop:Dol1}, $\HD_{\F}(p,q)=p-q+1$ holds whenever $q \geq \frac{p}{2}+1$. Our theorem applies only for $q \geq 200$ (as for any `relevant' pair $(p,q)$ we have $q \geq 200f(2p/q) \geq 200 \cdot 1$), and in this range, $(q>0.52p) \Rightarrow (q > 0.5p+1)$. Thus, either the above argument implies $\HD_{\F} <p-q+1$, or Proposition~\ref{Prop:Dol1} implies $\HD_{\F}=p-q+1$, and either way we are done.

\medskip \noindent \textbf{Case 2: $\F$ does not satisfy the $(p',q')$-property.} In this case, there exists a `bad' subfamily $S$ of size $p'=\frac{2p}{3}$ that does not contain an intersecting $q'$-tuple, and the family $\F \setminus S$ satisfies the $(\frac{p}{3},\frac{q}{2})$-property.

To pierce the family $\F \setminus S$ we use Lemma~\ref{Lem:App-Weak-(p,q)}. By the monotonicity of $f$, we have $\frac{q}{2} \geq 100f \left(\frac{2p}{q} \right) \geq 100 f\left(\frac{2 \cdot (p/3)}{q/2} \right)$, and thus, $\frac{q}{2} \geq T_{50}(\frac{p}{3})$. Hence, Lemma~\ref{Lem:App-Weak-(p,q)} implies
\[
\HD_{\F}(p/3,q/2)=\leq \frac{p}{3}-\frac{q}{2}+1+\frac{1}{50}\cdot \frac{p}{3}=0.34p-0.5q+1,
\]
whence $\F \setminus S$ can be pierced by $0.34p-0.5q+1$ points.

To pierce the `bad' subfamily $S$, we use Lemma~\ref{Lem:Cute}, which implies that $S$ can be pierced by
\[
\lfloor \frac{|S|+\lambda}{2} \rfloor \leq \frac{p}{3}+0.005p \leq 0.34p.
\]
Therefore, $\F$ can be pierced by $0.68p - 0.5q+1$ points. As in Case~1, we may argue that either $0.68p - 0.5q+1 < p-q+1$ and we are done, or $q \geq \frac{p}{2}+1$ and then we are done by Proposition~\ref{Prop:Dol1}. This completes the inductive step.

\medskip \noindent To conclude the proof, we need the induction basis. The idea is to show that for
\[
q_0=\min\{q: \mbox{ there exists a `relevant' pair (p,q) }\}
\]
(where `relevant' means a pair $(p,q)$ that belongs to the range covered by the theorem), for all relevant pairs $(p,q_0)$ we have $p \leq 2q_0-2$, and thus $\HD_{\F}(p,q_0)=p-q_0+1$ holds by Proposition~\ref{Prop:Dol1}. This is a sufficient basis, since in the inductive process, $q$ is decreased by 1 in each step, and so we will eventually reduce to $q=q_0$, for which the assertion holds. Note that we cannot move from $(p,q)$ to $(p',q-1)$ such that $p'<q-1$, since this would mean that the family contains an independent set of size $\geq p-q+2$; completing such a set to $p$ elements by adding $\leq q-2$ arbitrary elements, we obtain a subfamily of $\F$ of size $p$ without an intersecting $q$-tuple, a contradiction.

By Claim~\ref{Cl:T(p)}(2), $T_c(p)$ is increasing in $p$, and thus, if for some $q$ there exists a $p$ such that $q \geq T_c(p)$, then we also have $q \geq T_c(q)$. Hence, for each $q$, the smallest $p$ for which $(p,q)$ lies in the range covered by the theorem is $q$ itself. Consequently, the smallest $q$ for which there exists a `relevant' $(p,q)$ is equal to the smallest $p$ for which there exists a `relevant' $(p,q)$. Denote this value by $q_0$. By its minimality, we have $q_0-1<T_{100}(q_0-1)$. Therefore, by Claim~\ref{Cl:T(p)}(3a) we have $q_0<T_{100}(2q_0-1)$. (Note that in order to apply the claim, we need $c \geq \frac{1}{2}\log_{2(k-1)/k} 2$, where $k$ is a lower bound on $T_c(q_0-1)$. This indeed holds for $c=100$, as we can take $k=200$ as a lower bound, as mentioned above.) As $T_c(p)$ is increasing in $p$, this implies that $\{p:T_{100}(p)=q_0\} \subset \{q_0,q_0+1,\ldots,2q_0-2\}$. Therefore, by Proposition~\ref{Prop:Dol1}, we have $\HD_{\F}(p,q_0)=p-q_0+1$ for all `relevant' $(p,q_0)$. This completes the proof of the induction basis, and hence the proof of the theorem.
\end{proof}

\subsection{Proof of Theorem~\ref{Cor:Box}}
\label{app:sub:box}

We conclude this appendix with the simple deduction of Theorem~\ref{Cor:Box} from Theorem~\ref{Thm:Main}. Let us recall the statement of the theorem.

\medskip \noindent \textbf{Theorem~\ref{Cor:Box}.} Let $\F$ be a family of axis-parallel boxes in $\mathbb{R}^d$. Then $\HD_{\F}(p,q)=p-q+1$ holds for all $q> c\log^{d-1}(p)$, where $c$ is a universal constant.

\begin{proof}
The $(p,2)$-theorem for axis-parallel boxes in $\Re^d$~\cite{Kar91} yields $\HD_{\F}(p,2) \leq O(p \log_2^{d-1}(p))$, which means that $\HD_{\F}(p,2) \leq p f(p)$ holds for $f(p)=c' \log_2^{d-1}(p)$ (where $c'$ is a universal constant). For this $f(p)$, we have $T_{100}(p) \leq 200c' \log_2^{d-1}(p)$. Hence, the assertion of Theorem~\ref{Cor:Box} will follow from Theorem~\ref{Thm:Main}, once we verify that $f(p)$ satisfies the conditions of the theorem. The condition regarding $f'(p)$ is clearly satisfied: we have $f'(p)=c' (d-1) \log_2^{d-2}(p) \log_2 e/p \geq \log_2 e /p$, for all $p \geq 2$ and $c' \geq 1$. As for the condition regarding $f'(p)/f(p)$, we observe that in the proof of Theorem~\ref{Thm:Main}, this condition is applied only for values of $p$ for which there exists a `relevant' pair $(p,q)$, and thus, it is sufficient to show that it holds for all such values. We have
\[
\frac{f'(p)}{f(p)} = \frac{c' (d-1) \log_2^{d-2}(p) \log_2 e}{c' p \log_2^{d-1}(p)} = \frac{(d-1) \log_2 e}{p \log_2 p},
\]
and so we have to show that
\begin{equation}\label{Eq:App-Aux1}
\frac{(d-1) \log_2 e}{p \log_2 p} \leq \frac{5}{p}.
\end{equation}
This indeed holds in all the required range, since if there exists a `relevant' pair $(p,q)$ then $p \geq 200 \log^{d-1}(p)$, and thus, $\log_2 p \geq (d-1) \log_2 \log_2 p \geq d-1$, which clearly implies~\eqref{Eq:App-Aux1}. This completes the proof.
\end{proof}

\subsection{Proof of Observation~\ref{Obs:Dol}}
\label{app:obs-proof}

For the sake of completeness, we present in this subsection the proof of Observation~\ref{Obs:Dol}, due to Wegner~\cite{Weg65} and (independently) Dol'nikov~\cite{Dol72}. Let us recall its formulation.

\medskip \noindent \textbf{Observation~\ref{Obs:Dol}.}
Let $\F$ be a family that satisfies $\HD_{\F}(2,2)=1$, and put $\lambda=\nu(\F)$. Then
\[
\HD_{\F}(p,q) \leq \HD_{\F}(p-\lambda,q-1) + \lambda-1.
\]

\begin{proof}
The slightly weaker bound $\HD_{\F}(p,q) \leq \HD_{\F}(p-\lambda,q-1) + \lambda$ holds trivially, and does not even require the assumption $\HD_{\F}(2,2)=1$. Indeed, if $\mathcal{S}$ is a pairwise-disjoint subset of $\F$ of size $\lambda$, then $\F \setminus \mathcal{S}$ satisfies the $(p-\lambda,q-1)$-property, and thus, can be pierced by $\HD_{\F}(p-\lambda,q-1)$ points. As $\mathcal{S}$ clearly can be pierced by $\lambda$ points, we obtain $\HD_{\F}(p,q) \leq \HD_{\F}(p-\lambda,q-1) + \lambda$.

To get the improvement by 1, let $\mathcal{S}$ be a pairwise-disjoint subfamily of $\F$ of size $\lambda=\nu(\F)$ and let $T$ be a transversal of $\F \setminus \mathcal{S}$ of size $\HD_{\F}(p-\lambda,q-1)$. Take an arbitrary $x \in T$, and consider the subfamily $\mathcal{X}=\{A \in \F \setminus \mathcal{S}: x \in A\}$ (i.e., the sets in $\F \setminus \mathcal{S}$ pierced by $x$). By the maximality of $\mathcal{S}$, each $A \in \mathcal{X}$ intersects some $B \in \mathcal{S}$. Hence, we can write
$\mathcal{\mathcal{X}}=\cup_{B \in \mathcal{S}} \mathcal{X}_B$, where $\mathcal{X}_B = \{A \in \mathcal{X}: A \cap B \neq \emptyset\}$. Observe that for each $B$, the set $\mathcal{X}_B \cup \{B\}$ is pairwise-intersecting. Indeed, any $A,A' \in \mathcal{X}$ intersect in $x$, and all elements of $\mathcal{X}_B$ intersect $B$. Therefore, by the assumption on $\F$, each $\mathcal{X}_B \cup \{B\}$ can be pierced by a single point. Since $\mathcal{X} = \cup_{B \in \mathcal{S}} \mathcal{X}_B$, this implies that there exists a transversal $T'$ of $\mathcal{X} \cup \mathcal{S}$ of size $|\mathcal{S}|=\lambda$. Now, the set $(T \setminus \{x\}) \cup T'$ is the desired transversal of $\F$ with $\HD_{\F}(p-\lambda,q-1) + \lambda-1$ points.
\end{proof}

\section{Proof of Theorem~\ref{Thm:New-(p,2)2}}
\label{app:(p,2)}

In this appendix we prove Theorem~\ref{Thm:New-(p,2)2}, which provides $(p,2)$-theorems for families of compact convex sets that admit a $(2,2)$-theorem. In addition, for sake of completeness we present an example, due to Fon der Flaass and Kostochka~\cite{FdFK93}, of a set system (i.e., a hypergraph) with Helly number 2 that does not admit a $(3,2)$-theorem, thus showing that in general, the existence of a $(2,2)$-theorem does not imply the existence of a $(p,2)$-theorem.

\medskip Let us restate Theorem~\ref{Thm:New-(p,2)2} in a more precise form.

\medskip \noindent \textbf{Theorem~\ref{Thm:New-(p,2)2} (precise formulation).} Let $\F$ be a family of compact convex sets in $\Re^d$ such that $\HD_{\F}(2,2)=t$. Then:
\begin{enumerate}
\item We have
\[
\HD_{\F}(p,2) = \tilde{O} \left(4^{pd \cdot \frac{(p/t)-1}{(p/t)-d}}\right).
\]
In particular, $\F$ admits a $(p,2)$-theorem for piercing with a bounded number $s=s(p,d,t)$ of points.

\item If $d=2$, we have $\HD_{\F}(p,2) = O_t(p^8 \log^2 p)$.

\item If $d=2$ and the VC-dimension of $\F$ is $k$, then $\HD_{\F}(p,2) = O_{t,k}(p^4 \log^2 p)$.
\end{enumerate}

\noindent Two remarks are due at this point.
\begin{remark}
The difference between the general case (Part~(1) of the theorem) and the planar case (Parts~(2,3) of the theorem) looks surprisingly huge. We do not know whether any of these results are tight; however, the difference is well-explained by the proof method. While in the proof of Parts~(2,3) we use a Ramsey-type theorem for families of convex sets in the plane of Larman et al.~\cite{PT94} (Theorem~\ref{Thm:Larman} below) in which the `Ramsey number' $R(k)$ is polynomial in $k$, in the general case we have to resort to the classical Ramsey theorem in which $R(k)$ is exponential in $k$. This is in a sense necessary, since Tietze~\cite{Tie05} and (independently) Besicovitch~\cite{Bes47} showed that any graph can be represented as the intersection graph of a family of convex sets in $\Re^3$, which implies that no `Ramsey theorem for convex sets in $\Re^d$' for $d \geq 3$ can improve over the classical Ramsey theorem.
\end{remark}

\begin{remark}
As mentioned in the introduction, Matou\v{s}ek~\cite{MAT04} showed that families of sets with dual VC-dimension $k$ have fractional Helly number at most $k+1$. This allows deducing that such families admit a $(p,k)$-theorem, using the proof technique of the Alon-Kleitman $(p,q)$-theorem. This result (which applies in a much more general setting than Part~(3) of our theorem) does not imply our theorem, since it yields a $(p,2)$-theorem only for families with dual VC-dimension $1$, while our theorem applies whenever the VC-dimension is bounded.
\end{remark}

\medskip The proof of the theorem is a combination of three tools:

\medskip The first is a Ramsey-type theorem. Recall that the classical Ramsey theorem~\cite{Ram30} asserts that for any $k$, there exists $R(k)$ such that any graph on $R(k)$ vertices contains either a complete subgraph on $k$ vertices or an empty subgraph on $k$ vertices. Ramsey showed that $R(k) \leq {{2k-2}\choose{k-1}} \leq 4^k$. As the best currently known upper bound is not much lower, we use the upper bound $R(k) \leq 4^k$ for sake of simplicity.

The Ramsey theorem implies that any family of $R(k)$ sets contains either a pairwise-intersecting subfamily of size $k$ or a pairwise-disjoint subfamily of size $k$. Larman et al.~\cite{PT94} showed that for families of compact convex sets in the plane, a significantly better result can be achieved.
\begin{theorem}[\cite{PT94}]\label{Thm:Larman}
Let $\F$ be a family of $k \ell^4$ compact convex sets in the plane. Then $\F$ contains either $k$ pairwise-intersecting sets or $\ell$ pairwise-disjoint sets.
\end{theorem}

The second result is a quantitative bound for the Alon-Kleitman $(p,q)$-theorem obtained in~\cite[Theorem~1.3]{KST17}:
\begin{theorem}[\cite{KST17}]\label{Thm:KST}
Let $\F$ is a family of compact convex sets in $\Re^d$. Then
$$
\HD_{\F}(p,q) \leq \begin{cases}
\mathrm{(a)} \quad O\left(p^{d \cdot \frac{q-1}{q-d}} \log^{cd^3 \log d} p\right) = \tilde{O} \left(p^{d \cdot \frac{q-1}{q-d}} \right), & \textrm{ for all } q \geq d+1; \\
\mathrm{(b)} \quad \tilde{O} \left(p+ \left(\frac{p}{q}\right)^d \right) & \textrm{ if } q \geq \log p.\\
\end{cases}$$
$\mathrm{(c)}$ Furthermore, for $d=2$, the bound in~(b) can be replaced by $p-q+ O\left(\left(\frac{p}{q}\right)^2 \log^{2} \left(\frac{p}{q} \right) \right)$.
\end{theorem}

The third result is the $\epsilon$-net theorem for families with a bounded VC-dimension. Let us recall a few definitions.

For a set system $(U,R)$, where $U$ is a set of points and $R \subset \mathcal{P}(U)$ is a set of ranges (or alternatively, a hypergraph $(U,R)$ in which $U$ is the set of vertices and $R$ is the set of hyperedges), we say that a set $Y \subset U$ is \emph{shattered} by $R$ if every subset of $Y$ can be obtained as the intersection of $Y$ with some range $e \in R$. The \emph{VC-dimension} of $R$ is the maximal size of a set $Y$ that is shattered by $R$. For example, any set of three non-collinear points in the plane can be shattered by halfplanes, but no four points can be. Hence, the VC-dimension of halfplanes in the plane is 3. This notion was introduced by Vapnik and Chervonenkis~\cite{VC71}.

An $\epsilon$-net of $(U,R)$ is a subset $S \subset U$, such that any range $e \in R$ that contains at least $\epsilon$-fraction of the elements of $U$, intersects $S$.

The $\epsilon$-net theorem of Haussler and Welzl~\cite{HW} asserts the following:
\begin{theorem}[The $\epsilon$-net theorem,~\cite{HW}]\label{Thm:eps-net}
Let $(U, R)$ be a range space of VC-dimension $k$, let $A$ be a finite subset of $U$ and suppose $0 <\epsilon,\delta < 1$. Let $N$ be a set obtained by $m$ random independent draws from $A$, where
\[
m \geq \max \left(\frac{4}{\epsilon} \log \frac{2}{\delta}, \frac{8k}{\epsilon} \log \frac{8k}{\epsilon} \right).
\]
Then $N$ is an $\epsilon$-net for $A$ with probability at least $1-\delta$.
\end{theorem}
In particular, any family with VC-dimension $k$ admits an $\epsilon$-net of size $O_k(\frac{1}{\epsilon} \log \frac{1}{\epsilon})$.

\medskip Now we are ready to present the proof of the theorem.

\begin{proof}[Proof of Theorem~\ref{Thm:New-(p,2)2}] \textbf{Part (1).} As any family that satisfies the $(p,q)$-property, clearly satisfies the $(p',q)$-property for all $p' \geq p$ (if it contains at least $p'$ sets), and as we are not interested in constant factors, we may assume that $p$ is larger than any prescribed constant; in particular, we may assume $p \geq (d+1)t$. Let $\F$ be a family that satisfies the assumptions of the theorem and has the $(p,2)$-property. We claim that $\F$ satisfies the $(4^p,\lceil p/t \rceil)$-property.

To prove this, let $S$ be a subfamily of $\F$ of size $4^p$. We have to show that $S$ contains an intersecting $\lceil p/t \rceil$-tuple.

By the Ramsey theorem~\cite{Ram30}, either $S$ contains $p$ pairwise intersecting sets, or else it contains $p$ pairwise disjoint sets. The latter is impossible since $\F$ satisfies the $(p,2)$-property. Hence, $S$ contains a pairwise intersecting subfamily $T$ of size $p$. As $T$ satisfies the $(2,2)$-property, by the assumption on $\F$ it can be pierced by $t$ points. The largest among the subsets of $T$ pierced by a single point is of size $\geq \lceil p/t \rceil$, and so, $S$ contains an intersecting $\lceil p/t \rceil$-tuple, as asserted.

Since $\lceil p/t \rceil \geq d+1$ by assumption, we can apply Theorem~\ref{Thm:KST}(1) to $\F$ to deduce that
\[
\HD_{\F}(p,2) = \tilde{O} \left( \left(4^{p}\right)^{d \cdot \frac{\lceil p/t \rceil -1} {\lceil p/t \rceil-d}} \right) = \tilde{O} \left(4^{pd \cdot \frac{(p/t)-1}{(p/t)-d}} \right),
\]
completing the proof.


\medskip \noindent \textbf{Part (2).} As in the proof of Part~(1), we may assume that $p$ is sufficiently large so that $p/ \log p > 5t$. Let $\F$ be a family that satisfies the assumptions of the theorem and has the $(p,2)$-property. Applying the same argument as in Part~(1), with Theorem~\ref{Thm:Larman} instead of the Ramsey theorem, we deduce that $\F$ satisfies the $(p^5, \lceil p/t \rceil)$-property.

Since $p/t \geq \log(p^5)$ by assumption, we can apply to $\F$ Theorem~\ref{Thm:KST}(2) to deduce that
\[
\HD_{\F}(p,2) = O \left(p^5+ \left(\frac{p^5}{p/t}\right)^2 \left(\log \left(\frac{p^5}{p/t}\right)\right)^2\right) = O(p^8 \log^2 p),
\]
completing the proof.

\medskip \noindent \textbf{Part (3).}
From Theorem~\ref{Thm:Larman} we can deduce that for any $q$, $\F$ satisfies the $(tp^4q,q)$-property. By~\cite[Proposition~2.3]{KST17}, this implies that there exists a point that pierces a $\Omega(\frac{q}{(tp^4q)^{(q-1)/(q-2)}})$-fraction of the sets in $\F$. By~\cite[Proposition~2.6]{KST17}, this (in turn) implies that $\F$ can be pierced by $f(\beta)$ points, where $\beta=\Omega((tp^4q)^{-(q-1)/(q-2)})$ and $f(\beta)$ is the size of the minimal weak $\epsilon$-net guaranteed by the weak $\epsilon$-net theorem~\cite{ABFK} in the plane for $\epsilon=\beta$.

Since by assumption, the VC-dimension of $\F$ is bounded by $k$, $\F$ admits an $\epsilon$-net of size $O_k(\frac{1}{\eps} \log \frac{1}{\eps})$ by Theorem~\ref{Thm:eps-net}. Hence, we can replace the application of the weak $\eps$-net theorem in the above argument with an application of Theorem~\ref{Thm:eps-net}. Substituting $q=\log p$, which is easily seen to be (roughly) optimal, we obtain
\[
\HD_{\F}(p,2) = O_k \left( (tp^4 \log p)^{\frac{\log p-1}{\log p -2}} \cdot \log \left((tp^4 \log p)^{\frac{\log p-1}{\log p -2}} \right) \right) = O_{k,t}(p^4 \log^2 p),
\]
as asserted.
\end{proof}

\subsection{An example of a set system with Helly number 2 that does not admit a (3,2)-theorem}
\label{app:example}

The following example, presented by Fon der Flaass and Kostochka~\cite{FdFK93} in a different context, implies that in the abstract (i.e., non-geometric) setting, the existence of a $(2,2)$-theorem, and even Helly number 2, does not imply the existence of a $(p,2)$-theorem with a fixed number $f(p)$ of points that does not depend on the size of the family. The example uses a classical result of Erd\H{o}s~\cite{ES74} which asserts that for any $m \in \mathbb{N}$, there exists an $m$-chromatic triangle-free graph $G_m$ on $n(m)=O(m^2 \log^2 m)$ vertices.

As any graph on $n$ vertices can be represented as the intersection graph of a family of axis-parallel boxes in $\Re^{\lceil n/2 \rceil}$ (see~\cite{Rob69}), the \emph{complement graph} $\bar{G}_m$ can be represented as the intersection graph of some family $\F$ of axis-parallel boxes in $\Re^{n'}$, for $n'=\lceil \frac{n(m)}{2} \rceil =O(m^2 \log^2 m)$. The family $\F$ has Helly number 2 (like any family of axis-parallel boxes). It satisfies the $(3,2)$-property, since if some three elements of $\F$ are pairwise disjoint, then the intersection graph of $\F$ contains an empty triangle, and this cannot happen since the intersection graph $\bar{G}_m$ is the complement of a triangle-free graph. On the other hand, $\F$ cannot be pierced by less than $m$ points, as any transversal of $\F$ of size $k$ induces a partitioning of the vertices of $\bar{G}_m$ into into $k$ cliques, which in turn yields a $k$-coloring of $G_m$ (which was assumed to be $m$-chromatic). Therefore, we have $\HD_{\F}(3,2) \geq m$ although the Helly number of $\F$ is 2.

\section{Discussion and open problems}

A central problem left for further research is whether Theorem~\ref{Thm:Main} which allows leveraging a $(p,2)$-theorem into a $(p,q)$-theorem, can be extended to the cases $\HD_{\F}(p,2) = pf(p)$ where $f(p) \ll \log p$ or $f(p)$ being super-polynomial in $p$. It seems that super-polynomial growth rates can be handled with a slight modification of the argument (at the expense of replacing $T_{100}(p)$ with some worse dependence on $p$). For sub-logarithmic growth rate, it seems that the current argument does not work, since the inductive step requires the packing number of $\F$ to be extremely small, and so, Lemma~\ref{Lem:Cute} allows reducing the piercing number of the `bad' family $S$ only slightly, rendering Lemma~\ref{Lem:App-Weak-(p,q)} insufficient for piercing $\F$ with $p-q+1$ points in total.

Extending the method for sub-logarithmic growth rates will have interesting applications. For instance, it will immediately yield a tight $(p,q)$-theorem for all $q=\Omega(\log \log p)$ for families of axis-parallel boxes in which for each two intersecting boxes, a corner of one is contained in the other, following the work of Chudnovsky et al.~\cite{CSZ17+}.

\medskip Another open problem is whether the method can be extended to families $\F$ that admit a $(2,2)$-theorem, but satisfy $\HD_{\F}(2,2)>1$. A main obstacle here is that in this case, Observation~\ref{Obs:Dol} does not apply, and instead, we have the bound $\HD(p,q) \leq \HD(p-\lambda,q-1)+\lambda$. While the bound is only slightly weaker, it precludes us from using the inductive process of Wegner and Dol'nikov, as in each application of the inductive step we have an `extra' point.

%

\end{document}